\documentclass[12pt]{article}
\usepackage{graphicx,amsmath,amssymb,amsthm,rotating,stmaryrd,datetime,pifont,booktabs,enumitem}
\usepackage[mathscr]{eucal}
\usepackage{framed}

\setlist{nosep}

\usepackage[usenames,dvipsnames]{color}

\newtheorem{theorem}{Theorem}

\newtheorem{result}[theorem]{Result}
\newtheorem{lemma}[theorem]{Lemma}
\newtheorem{corollary}[theorem]{Corollary}

\newtheorem{remark}[theorem]{Remark}

\usepackage[usenames,dvipsnames]{color}
\newcommand\red[1]{{\color{red} #1}}

\definecolor{beaublue}{rgb}{0.74, 0.83, 0.9}
\definecolor{cadetblue}{rgb}{0.37, 0.62, 0.63}

\definecolor{cambridgeblue}{rgb}{0.64, 0.76, 0.68}

\definecolor{beaver}{rgb}{0.62, 0.51, 0.44}
\definecolor{bistre}{rgb}{0.24, 0.17, 0.12}
\definecolor{darkbrown}{rgb}{0.4, 0.26, 0.13}

\definecolor{eggplant}{rgb}{0.38, 0.25, 0.32}

\definecolor{cadmiumgreen}{rgb}{0.0, 0.42, 0.24}
\definecolor{lightgrey}{rgb}{0.8, 0.8, 0.8}
\definecolor{grey}{rgb}{0.5, 0.5, 0.5}

\usepackage{fancybox}
\newcommand{\B}{{\textup{\bfseries{B}}}}

\newcommand{\Q}{\mathcal{Q}}

\renewcommand{\H}{\mathcal{H}}
\renewcommand{\P}{\mathcal{P}}
\renewcommand{\S}{\mathcal{S}}

\newcommand{\K}{\mathcal{K}}
\newcommand{\N}{\mathcal{N}}
\newcommand{\E}{\mathcal{E}}
\newcommand{\R}{\mathcal{R}}
\newcommand{\C}{\mathcal{C}}
\newcommand{\D}{\mathcal{D}}

\newcommand{\li}{\ell_\infty}
\newcommand{\si}{\Sigma_\infty}

\newcommand{\F}{\mathcal{F}}
\newcommand{\PG}{{\textup{PG}}}

\newcommand{\Fqq}{\mathbb{F}_{q^2}}
\newcommand{\Fqqq}{\mathbb{F}_{\hspace*{-1mm}{q^{3}}}}
\newcommand{\Fqqqq}{\mathbb{F}_{\hspace*{-1mm}{q^{4}}}}

\newcommand{\Fq}{\mathbb{F}_{\hspace*{-.5mm}q}}

\definecolor{shadecolor}{rgb}{0.9, 0.9, 0.9}

\begin{document}

\title{The intersection of two linear sets of rank 3 in $\PG(1,q^3)$}
\author{S.G. Barwick\footnote{School of Mathematical Sciences, Adelaide University, Adelaide, 5005, Australia. susan.barwick@adelaide.edu.au}, 
Alice M.W. Hui\footnote{School of Mathematical and Statistical Sciences, Clemson University, USA. aliceh2@clemson.edu} 
and Wen-Ai Jackson\footnote{School of Mathematical Sciences, Adelaide University, Adelaide, 5005, Australia. wen.jackson@adelaide.edu.au}}
\date{}
\maketitle





\begin{abstract} This article studies the intersection of  two linear sets of rank 3 of $\PG(1,q^3)$. In particular, all possible intersection sizes are determined, the structure of the largest intersection sets is  described, and existence results are given. \end{abstract}

 
 \section{Introduction}
 
 The intersection problem for two linear sets $\D_1,\D_2$ in $\PG(r,q^n)$ was described in \cite{ZZ}  as the following two questions: (1) do $\D_1,\D_2$ meet in at least one point? and (2) if $|\D_1\cap \D_2|>0$, then what is the size of $\D_1\cap \D_2$? While this is in general a hard question, the intersection problem is relevant to several geometrical problems, including multiple blocking sets  \cite{BBL,lunardon}, KM-arcs \cite{DBVDV},
and  finite semifields \cite{CPT,SVDVV}. See \cite{lavr15,polverino} for more details.
 
The intersection problem has been investigated in several cases, including: two $\Fq$-linear sets of rank 3 in $\PG(2,q^3)$, see \cite{ferr03}; two subgeometries, see \cite{DD2008}; an $\Fq$-line with an $\Fq$-linear set in $\PG(1,q^n)$, see \cite{lavr10,pepe11}; two linear sets of rank $3$ in $\PG(1,q^n)$, see \cite{lavr10}; two scattered linear sets of rank $n+1$ in $\PG(2,q^n)$, see \cite{DD2014}; a scattered linear set of rank $3n$ with either a line or a plane in $\PG(2n-1,q^3)$, see \cite{lavr13}; two clubs in $\PG(1,q^n)$, see \cite{SVDVV}; two linear sets of rank $n$ in $\PG(1,q^n)$, see \cite{ZZ}.

This article  looks at two linear sets of rank 3 in $\PG(1,q^3)$, there are two such sets, namely clubs and maximum scattered linear sets. The following is known in $\PG(1,q^3)$. The first intersection problem question is looked at in 
\cite[Cor 6.6]{SVDVV}, which shows that if $q$ is odd, then there exists two clubs $\D_1,\D_2$ in $\PG(1,q^3)$ with $|\D_1\cap\D_2|=0$.
Further \cite[Remark 3.13]{SVDVV} conjectures that if $q$ even, then   $|\D_1\cap\D_2|>0$ for all clubs $\D_1,\D_2$. 
The second intersection problem is looked at in \cite{ferr03} which gives a bound on the size of the intersection. 

\begin{result}\cite[Lemmas 2.3, 2.4, 2.5]{ferr03} \label{thm-FS}
Two $\Fq$-linear sets of rank 3 in $\PG(1,q^3)$  meet in at most $2q+2$ points.  
\end{result}

Further, \cite[Remark 24]{lavr10} gives an example of two maximum scattered linear sets in $\PG(1,q^3)$ that meet in exactly $2q+2$ points, showing that the  bound in Result~\ref{thm-FS} is tight for the case of two maximum scattered linear sets. We will show that the bound is not tight when one of the linear sets is a club. In particular, we show that  in  $\PG(1,q^3)$,  two clubs meet in at most $2q$ points, and a club and a maximum scattered linear set meet in at most $2q+1$ points. Moreover, both these bounds are tight.

This article 
 determines all possibilities for the size of the intersection $\D_1\cap\D_2$ of  two linear sets of rank 3 in $\PG(1,q^3)$. The case when $\D_1,\D_2$ are both maximum scattered linear sets is looked at in Theorem~\ref{113}; when $\D_1$ is a maximum scattered linear set and $\D_2$ is a club is looked at in Theorem~\ref{114}; and when $\D_1$ and $\D_2$ are both clubs is looked at in Theorem~\ref{116}.  Further, in each case we describe the structure of the largest intersections, and look at the conditions whereby the intersection contains an $\Fq$-line. Moreover, in  each case, the main theorem is followed by a corollary which gives some existence results.

%

\section{Background and preliminaries}

\subsection{The BB-representation of $\PG(2,q^3)$ in $\PG(6,q)$}\label{sec2.1}

A \emph{$2$-spread} 
 of $\PG(5,q)$ is a set of planes that partition the points of $\PG(5,q)$. A $2$-spread 
$\S$ is called \emph{regular} (or Desarguesian) if for any three planes in $\S$, the $2$-regulus containing them is contained in $\S$. See \cite[Section 25.6]{hirs91} for
more information.
A regular $2$-spread of $\PG(5,q)$ can be constructed as follows, see \cite{segre}. Embed $\PG(5,q)$ in $\PG(5,q^3)$ and let $g$ be a line of
$\PG(5,q^3)$ disjoint from $\PG(5,q)$. The set of planes in  $\{\langle P_i,P_i^q,P_i^{q^2}\rangle\cap\PG(5,q)\,|\, P_i\in g\}$ partition the points of $\PG(5,q)$. Moreover, these planes form a
regular spread $\S$ of $\PG(5,q)$. The lines $g$, $g^q$, $g^{q^2}$ are called the (conjugate
skew) {\em transversal lines} of the spread $\S$. Conversely, given a regular 2-spread
in $\PG(5,q)$,
there is a unique set of three (conjugate skew) transversal lines in the cubic extension  $\PG(5,q^3)$ that generate
$\S$ in this way.

The linear representation of a finite
translation plane $\P$ of order $q^t$ in $\PG(2t,q)$ was developed independently by
Andr\'{e}~\cite{andr54} and Bruck and Bose
\cite{bruc64,bruc66}. We will use the vector space construction as developed by Bruck and Bose, and describe the representation of $\PG(2,q^3)$ in $\PG(6,q)$. 
Let $\si$ be a hyperplane of $\PG(6,q)$ and let $\S$ be a 2-spread
of $\si$.  Consider the following incidence
structure:
the \textsl {points} of $\mathcal A$ are the points of $\PG(6,q)\setminus\si$; the \textsl {lines} of $\mathcal A$ are the 3-spaces of $\PG(6,q)$ that contain
  an element of $\S$ and are not contained in $\si$; and \textsl {incidence} in $\mathcal A$ is induced by incidence in
  $\PG(6,q)$.
Then the incidence structure $\mathcal A$ is an affine plane of order $q^3$. We
can complete $\mathcal A$ to a projective plane $\mathcal P$; the points on the line at
infinity $\li$ have a natural correspondence to the elements of the 2-spread $\S$.
The plane $\mathcal P$ is Desarguesian iff the spread $\S$ is a regular (Desarguesian) spread. We call this the $\PG(6,q)$ BB-representation of $\PG(2,q^3)$. 

\subsection{The BB-representation of $\PG(1,q^3)$ in $\PG(3,q)$}\label{sec2.2}

We use the BB-representation of $\PG(2,q^3)$ in $\PG(6,q)$ given in Section~\ref{sec2.1} to define our notation for the BB-representation of $\PG(1,q^3)$ in $\PG(3,q)$. Let $\ell=\PG(1,q^3)$ and chose a point at infinity $P_\infty\in\ell$. Embed 
  $\ell$ as a line of $\PG(2,q^3)$ with $P_\infty=\ell\cap\li$. Consider the BB-representation of $\PG(2,q^3)$ in $\PG(6,q)$. The point $P_\infty$ corresponds to a spread plane denoted $\pi_\infty$ and the line $\ell$ corresponds to a 3-space containing $\pi_\infty$. Associated with $\S$ are three transversal lines $g,g^q,g^{q^2}$ in the cubic extension of $\si$.  
  So associated with $\pi_\infty$ are the three conjugate points which are the intersection of the cubic extension of $\pi_\infty$ with the three transversal lines $g,g^q,g^{q^2}$. We called these three points the \emph{three conjugate points defining the spread plane $\pi_\infty$}.
 We call this setting the \emph{$\PG(3,q)$ BB-representation of $\ell=\PG(1,q^3)$ with point at infinity $P_\infty$.} We call $\pi_\infty$ the \emph{plane at infinity} in this representation. 

The representation of $\Fq$-lines of $\PG(1,q^3)$ in the $\PG(3,q)$ BB-representation is determined in \cite{BJ-2012} and we need the following result. (Note that the more general case of the representation of $\Fq$-lines of $\PG(2,q^t)$ in $\PG(2t,q)$ is determined in \cite{RSV}.)

\begin{result}\cite[Thm 2.3, 2.5]{BJ-2012} \label{125}
Consider the $\PG(3,q)$ BB-representation of $\PG(1,q^3)$, with point at infinity $P_\infty$, and denote the plane at infinity by $\pi_\infty$.
\begin{enumerate}
\item
The set $b$ is an $\Fq$-line  that contains the point $P_\infty$  iff the  BB-representation of $b$ is a line not contained in $\pi_\infty$.
\item The set $b$ is an $\Fq$-line  that does not contain the point $P_\infty$  iff the  BB-representation of $b$ is a   twisted cubic  whose cubic extension contains the three conjugate points defining the spread plane $\pi_\infty$.
\end{enumerate}
\end{result}

\subsection{Linear sets of rank 3 in $\PG(1,q^3)$} \label{sec2.3}

In this article we are only concerned with linear sets of rank 3 in $\PG(1,q^3)$. See \cite{lavr15,polverino} for a general introduction to linear sets. 
In the $\PG(6,q)$ BB-representation described above, the line at infinity $\li\cong\PG(1,q^3)$ is represented by a regular (Desarguesian) 2-spread $\S$ in $\PG(5,q)$. This representation arises from the field reduction map   which maps a point $P\in\PG(1,q^3)$ to a  plane   of the spread $\S$ in $\PG(5,q)$. That is, the points of $\PG(1,q^3)$ are in one to one correspondence with the planes of the spread $\S$. If $\pi$ is a plane of $\PG(5,q)$, we let $\B(\pi)$ denote the set of points of $\PG(1,q^3)$ satisfying $X\in\B(\pi)$ iff in $\PG(5,q)$, the  plane of $\S$ corresponding to $X$ meets the plane $\pi$. 
 A set $\D$ in $\PG(1,q^3)$ is called a \emph{linear set of rank 3} iff there is a plane $\pi$ of $\PG(5,q)$ which is not contained in the spread $\S$ such that $\D=\B(\pi)$. 
 
 There are two types of planes of $\PG(5,q)$ that are not in the spread $\S$, so there are two types of linear sets of rank 3. Firstly, 
 suppose $\pi$ is a plane of $\PG(5,q)$ that  meets a unique spread element $\alpha$ in a line, so $\pi$ meets $q^2$ further spread elements in a point. Then $\B(\pi)$ has size $q^2+1$ and is called a \emph{club}; moreover, the point of the club corresponding to $\alpha$ is called the \emph{head} of the club.  A club in $\PG(1,q^3)$ has a unique head, see \cite[Thm 4.1]{BJtgt1}; that is, if $\beta$ is a  plane distinct from $\pi$ with $\B(\beta)=\B(\pi)$, then the plane $\beta$ also meets $\alpha$ in a line.
The second type of linear set of rank 3 arises as follows. Suppose 
  $\pi$ is a plane of $\PG(5,q)$ that meets $q^2+q+1$ planes of $\S$, then $\B(\pi)$   has size $q^2+q+1$ and  is called a \emph{maximum scattered linear set of $\PG(1,q^3)$}.

It is proved in \cite[Thm 5]{lavr10} 
that all clubs in $\PG(1,q^3)$ are projectively equivalent; and all maximum scattered linear sets in $\PG(1,q^3)$ are projectively equivalent.
The article \cite{lavr10}  studies the relationship between $\Fq$-lines and linear sets of rank 3 in $\PG(1,q^3)$, and we summarise the following useful results.  

\begin{result}\cite{lavr10}\label{333}
\begin{enumerate}
\item  An $\Fq$-line meets a linear set  of rank $3$ of $\PG(1,q^3)$  in $0,1,2,3$ or $q+1$ points. 
\item Let $\D$ be  a club in $\PG(1,q^3)$, $q\geq3$. Then $\D$ contains $q^2+q$ $\Fq$-lines, these all contain the head of $\D$. Moreover any two non-head points of $\D$ lie on a unique $\Fq$-line of $\D$. 
\item Let $\D$ be  a maximum scattered linear set in $\PG(1,q^3)$, $q\geq3$. Then $\D$ contains $2(q^2+q+1)$ $\Fq$-lines, these  belong to two families of size $q^2+q+1$. Any two points of $\D$ lie in exactly two $\Fq$-lines of $\D$, one from each family. Two distinct $\Fq$-lines in the same family meet in a unique point; two $\Fq$-lines from different families meet in either $0,$ $1$ or $2$ points. 
\end{enumerate}
\end{result}

Note that if $q=2$, then every set of 3 points of $\PG(1,q^3)$ forms an $\Fq$-line, so the condition of $q\geq3$ in this result is necessary. In particular, if $q=2$, then a club contains $\Fq$-lines that do not contain the head; and a maximum scattered linear set does not have two families of $\Fq$-lines. 

Let $E,F$ be two points of $\PG(1,q^3)$ and let $G$ be the collineation group acting on the points of $\PG(1,q^3)$ that fixes $E,F$ pointwise. The group $G$ is isomorphic to the multiplicative group of $\Fqqq\setminus\{0\}$, so $G$ has a unique Singer subgroup $H$ of order $q^2+q+1$. The orbits of $H$ partition the points of $PG(1,q^3)$ into $q-1$ sets  of size $q^2+q+1$ which are maximum scattered linear sets. Further, given a maximum scattered linear set $\D$ of $\PG(1,q^3)$, then there is a unique pair of points $E,F$ that give rise to $\D$ in this way, the points $E,F$ are called the \emph{carriers} of $\S$, see \cite{BJ-ext1}. 

\subsection{Linear sets in the $\PG(3,q)$ BB-representation}

While the natural setting for defining linear sets of rank 3 of $\PG(1,q^3)$ is in the $\PG(5,q)$ field reduction setting used in Section~\ref{sec2.3}, we can also look at them in the $\PG(3,q)$ BB-representation defined in Section~\ref{sec2.2}.
The next result looks at a linear set of rank 3 of $\ell=\PG(1,q^3)$  and describes the corresponding set in the $\PG(3,q)$ BB-representation. 
This result is proved in \cite{DDVDV} (see Prop 3.3, Thm 3.8 and Thm 4.6) using a direct geometric argument. We give an alternative  short proof that utilizes the algebraic structure of Sherk surfaces  \cite{sher86}. Moreover, our proof also shows that  part 2 does not need the condition $q\geq4$ given in \cite[Thm 3.8]{DDVDV}; and  part 3 does not need the condition $q\geq5$ given in \cite[Thm 4.6]{DDVDV}.

\begin{theorem}  \label{001}
Let $\D$ be an $\Fq$-linear set of rank 3 of $\PG(1,q^3)$. Consider the $\PG(3,q)$ BB-representation of $\PG(1,q^3)$ with point at infinity $P_\infty\in\D$, and denote the plane at infinity by $\pi_\infty$. 
\begin{enumerate}
\item  The set $\D$ is a club with head $P_\infty$ iff the  BB-representation of $\D$ is a plane distinct from $\pi_\infty$.
\item  The set $\D$ is a club with head distinct from $P_\infty$ iff the   BB-representation of $\D$ is a  quadratic cone that meets $\pi_\infty$ in a   conic  whose cubic extension contains the three conjugate points defining the plane $\pi_\infty$. Further the vertex of the cone represents the head of the club.
\item The set $\D$ is a maximum scattered linear set iff  the   BB-representation of $\D$ is a hyperbolic quadric  that meets $\pi_\infty$ in a   conic  whose cubic extension contains the three conjugate points defining the plane $\pi_\infty$. \end{enumerate}
\end{theorem}

\begin{proof} Let $\D$ be a club of $\PG(1,q^3)$  that contains the point $P_\infty$. By
\cite[Thm 8.1]{BJtgt1}, in the $\PG(3,q)$ BB-representation,  $\D$ is represented by a Sherk surface of size $q^2+1$; and by Sherk \cite{sher86} this is either a  plane distinct from $\pi_\infty$, or a quadratic cone that meets $\pi_\infty$ in a    non-degenerate  conic  whose cubic extension contains the three conjugate points defining the plane $\pi_\infty$. 

We consider the $\PG(5,q)$ field reduction representation of $\PG(1,q^3)$, so the point $P_\infty\in\PG(1,q^3)$ corresponds to a plane $\pi_\infty$. A 3-space $\Gamma$  that contains $\pi_\infty$ meets this $\PG(5,q)$ field reduction representation in  a $\PG(3,q)$ BB-representation of $\PG(1,q^3)$. It follows from the definition of a club that if 
$\pi$ is a plane in $\PG(3,q)$ that is distinct from $\pi_\infty$, then $\pi$ corresponds to a club of $\PG(1,q^3)$ that has head $P_\infty$. So we have shown that a club with head $P_\infty$ has a  $\PG(3,q)$  BB-representation which is either an affine plane or a certain type of quadratic cone; further an affine plane of $\PG(3,q)$ represents a club of $\PG(1,q^3)$ with head $P_\infty$. To complete the proof of part 1, we count. (Recall that a club in $\PG(1,q^3)$ has a unique head by \cite[Thm 4.1]{BJtgt1}.) The number of clubs with head $P_\infty$ is $q(q^2+q+1)$ (see \cite[Thm 5.2]{BJtgt1}). This matches the number of  affine planes in $\PG(3,q)$, distinct from $\pi_\infty$, which is also $q^3+q^2+q$. This completes the proof of part 1.

As a club in $\PG(1,q^3)$ has a unique head, it follows that a club with head distinct from $P_\infty$ is represented in $\PG(3,q)$  by a  quadratic cone that meets $\pi_\infty$ in a   conic  whose cubic extension contains the three conjugate points defining the plane $\pi_\infty$. 
We prove the converse by counting. Given a point $Y\neq P_\infty$ in $\PG(1,q^3)$,  the number of clubs with head $Y$ that contain $P_\infty$ is $q^2+q+1$ (to show this, count  incident pairs $(X,\D)$ where $X$ is a point distinct from $Y$ and $\D$ is a club with head $Y$). On the other hand, it is straightforward to show there are $q^2+q+1$ non-degenerate conics in $\pi_\infty$ whose cubic extension contains the three conjugate points defining $\pi_\infty$.  Hence given an affine point $V\notin\pi_\infty$, there are $q^2+q+1$ quadratic cones with vertex $V$ that meet $\pi_\infty$ in a conic  whose cubic extension contains the three conjugate points defining the plane $\pi_\infty$. This completes the proof of the iff statement in part 2.  

Finally we show that the vertex of the quadratic cone corresponds to the head of the club. If $q\geq3$, then by Result~\ref{333}, the head of a club $\D$ lies in $q^2+q$ $\Fq$-lines that are contained in $\D$; and every other point of $\D$ lies in $q+1$ $\Fq$-lines that are contained in $\D$. By Result~\ref{125}, in the $\PG(3,q)$ BB-representation, these $\Fq$-lines corresponds to either affine lines, or certain twisted cubics that lie in the quadratic cone. By \cite[Lemma 21.1.7]{Hirsch2}, all twisted cubics contained in a quadratic cone contain the vertex. 
It follows that the vertex of the cone corresponds to the head of the club. When $q=2$, every 3 points form an $\Fq$-line, so we cannot geometrically identify the head in this way. However, an algebraic proof can be used: using the notation from Sherk \cite{sher86} we briefly note that the map $\Delta$  maps the infinite plane $\pi_\infty$ to the point $(0,0,0,1)$; and maps the plane $\S(0,0,1,0)$ (which corresponds to a club with head $P_\infty=(1,0)$) to the quadratic cone $\S(0,1,0,0)$ which has vertex $(0,0,0,1)$  (this corresponds to a club with head $(0,1)$).  We can conclude that the head of the club corresponds to the vertex of the cone.

To prove part 3, let $\D$ be a maximum scattered linear set of $\PG(1,q^3)$ that contains $P_\infty$. By 
 \cite[Remark 3.7]{BJ-2012}, in the $\PG(3,q)$ BB-representation,  $\D$ is represented by a  Sherk surface of size $q^2+q+1$; and by Sherk \cite{sher86} this is a hyperbolic quadric that meets $\pi_\infty$ in a non-degenerate  conic  whose cubic extension contains the three conjugate points defining the plane $\pi_\infty$. To show the converse, we count. The number of hyperbolic quadric that meets $\pi_\infty$ in a   conic  whose cubic extension contains the three conjugate points defining the plane $\pi_\infty$ is $\frac{1}{2}q^3(q^3-1)$ (see \cite[Prop 4.4]{DDVDV}). We next count the number of maximum scattered linear sets in $\PG(1,q^3)$ that contain the point $P_\infty$, the count in 
 \cite[Prop 4.5]{DDVDV} needs $q\geq5$, so we use a different technique to remove the condition on $q$. 
 We recall from Section~\ref{sec2.3} that a maximum scattered linear set has two unique carrier points, and given two points $E,F$, there are exactly $q-1$ (pairwise disjoint) maximum scattered linear sets with carriers $E,F$. So the number of maximum scattered linear sets that contain $P_\infty$ is equal to the number of choices for two carrier points (distinct from $P_\infty$) which is ${{q^3}\choose{2}}$. 
%
%
%
We conclude that part 3 holds for all $q$. 
\end{proof}

 We note that if  $\D$ is an $\Fq$-linear set of rank 3 of $\PG(1,q^3)$ with $P_\infty\notin\D$, then Sherk \cite{sher86} shows that in the  $\PG(3,q)$ BB-representation, $\D$ corresponds to  a cubic surface.

Recall from Result~\ref{333} that if $q\geq3$, a maximum scattered linear set in $\PG(1,q^3)$ contains two families of $\Fq$-lines. The next result describes these families in the $\PG(3,q)$ BB-representation.  

\begin{lemma}\label{117}
Let $\D$ be a maximum scattered linear set in $\PG(1,q^3)$, $q\geq3$, and denote the two families of $\Fq$-lines in $\D$ by $\F_1$, $\F_2$. 
Let $\H$ be the hyperbolic quadric corresponding to $\D$ in the $\PG(3,q)$ BB-representation. Then we can label the  two reguli of $\H$   by $\R_1,\R_2$ such that the following hold.
\begin{enumerate}
\item The set $b$, with $P_\infty\in b$, is an  $\Fq$-line in family $\F_1$   iff the BB-representation of $b$ is  a line of the  regulus $\mathcal R_1$. 
\item The set $b$, with $P_\infty\notin b$, is an  $\Fq$-line in family $\F_1$    iff the BB-representation of $b$ is   a  twisted cubic curve $\N_1$ of $\H$  such  that: the cubic extension of $\N_1$ contains the three conjugate points defining the plane $\pi_\infty$;   the lines of $\mathcal R_1$ are non-tangential unisecants of $\N_1$; and the lines of the $\R_2$ are chords of $\N_1$. 
\end{enumerate}
The $\Fq$-lines in $\F_2$  can be similarly described. 

\end{lemma}

\begin{proof} By Result \ref{125}, the lines of $\R_1$ and $\R_2$ correspond to $\Fq$-lines through $P_\infty$; and  the $\Fq$-lines in $\D_i$ containing $P_\infty$ correspond to generator lines of $\H$. If $\ell_1\in\R_1$ and $\ell_2\in\R_2$, then $\ell_1,\ell_2$ correspond to $\Fq$-lines of $\D$ which meet in two points, so by Result~\ref{333}(3), they belong to different families. Hence the lines of $\R_i$ correspond to $\Fq$-lines  of $\D$  through $P_\infty$ in family $\F_i$, $i=1,2$ and conversely.

By Result \ref{125}, the $\Fq$-lines contained in $\D$ that do not contain $P_\infty$ correspond exactly to  twisted cubics contained in $\H$ whose cubic extension contains the three conjugate points defining the plane $\pi_\infty$.
Let $\N$ be a twisted cubic contained in $\H$, then by \cite[Lemma 21.1.7] {Hirsch2}, all the lines of one regulus of $\H$ are chords of $\N$ and all the lines of the other regulus of $\H$  are  non-tangential unisecants of $\N$. As $\Fq$-lines contained  in the same family of $\D$ meet in exactly one point, and $\Fq$-lines in different families meet in two points (see Result~\ref{333}(3)), the correspondence in part 2 holds. Similarly for $\Fq$-lines in family $\F_2$. 
\end{proof}

We complete the preliminaries with a short existence result on the intersection of quadratic cones.

\begin{lemma}\label{118}  
 In $\PG(3,q)$, there exists two quadratic cones  whose intersection is a non-degenerate conic iff $q$ is even. 
\end{lemma}

\begin{proof}
%
%
%
Let $\Q_1,\Q_2$ be
two quadratic cones with vertices $V_1,V_2$ respectively. Suppose the intersection $\Q_1\cap\Q_2$ contains a   conic $\C$. Denote the plane containing $\C$ by $\pi$.  
Then $\Q_1,\Q_2$ meet exactly in $\C$ if and only if
when projecting from $V_1$ onto $\pi$, 
the generators of $\Q_2$ are projected to lines of $\pi$ which are tangent to $\C$. This latter condition only happens when $V_1V_2\cap \pi$ is the nucleus of $\C$, which is not possible when $q$ is odd. 
When $q$ is even, $\Q_1,\Q_2$ meet exactly in $\C$ when $V_1,V_2$ and the nucleus of $\C$ are collinear. Given a conic $\C$, it is always possible to choose $V_1,V_2$ collinear with the nucleus of $\C$, so we can always   construct  two quadratic cones that meet in exactly one conic. 
\end{proof}

\section{Main results}

\subsection{The intersection of two maximum scattered linear sets}

%
 
We  prove the following about the intersection of two maximum scattered linear sets  in $\PG(1,q^3)$.

\begin{theorem}\label{113} Let $\D_1,\D_2$ be two distinct maximum scattered linear sets in $\PG(1,q^3)$, then the following hold.
\begin{enumerate}
\item\label{113a} $|\D_1\cap\D_2|\in\{0,1, 2q,2q+1,2q+2\} \cup  \{n\,|\, q-2\sqrt{q} +1\leq n\leq q+ 2 \sqrt{q}+1  \mbox{ and } n \neq2,3 \}  $. 
\item\label{113b} If $q\geq3$ and $\D_1\cap\D_2$ contains an $\Fq$-line, then $\D_1\cap\D_2$ is two $\Fq$-lines. These two $\Fq$-lines belong to different families in $\D_1$ (and belong to different families in $\D_2$).
\item\label{113c} If    $q\geq5$ and  $|\D_1\cap\D_2|\in\{2q,2q+1,2q+2\} $ with  $|\D_1\cap\D_2|\neq10$,  then $\D_1\cap\D_2$  is two $\Fq$-lines.
\end{enumerate}
\end{theorem}

\begin{proof}
Suppose $|\D_1\cap\D_2|>0$. Without loss of generality, assume $P_\infty\in\D_1\cap\D_2$ and 
consider the BB-representation of $\PG(1,q^3)$ in $\PG(3,q)$, with the plane at infinity denoted $\pi_\infty$.  
By Theorem~\ref{001}, for $i=1,2$, $\D_i$ corresponds  in $\PG(3,q)$ to a  hyperbolic quadric $\H_i$ that meets $\pi_\infty$ in a conic whose cubic extension contains the three conjugate points defining the plane $\pi_\infty$, denote these points by  $P,P^q,P^{q^2}$.   

We now consider the intersection $\S=\H_1 \cap \H_2$. We use the classification given by Bruen and Hirschfeld \cite{BH} which is briefly summarised as follows.  Let $\Q_1,\Q_2$ be two quadrics in $\PG(3,q)$ with quadratic forms $F,G$ respectively. Then the \emph{pencil} of quadrics $\Gamma=F+\lambda G$,  $\lambda \in\Fq\cup\{\infty\}$ has \emph{base} a curve  $\K$ of degree 4. That is, the quadrics in the pencil $\Gamma$ pairwise meet in $\K$. In the case where at least one of the quadrics in the pencil is non-singular (which is our situation as $\H_1,\H_2$ are non-singular hyperbolic quadrics), \cite{BH} lists all the possibilities for the base. We use the notation for cases given in \cite{BH}, which considers the different possible ways the base can reduce. In particular, $\S$ is the set of points occurring as one of the following bases:
\begin{enumerate}
\item four lines (cases 1a, 1b, 1c, 1e, 1d, 1f, 1g, 1h, 1i), 
\item two lines and a non-degenerate conic (cases 2a, 2b, 2c, 2d), 
\item two non-degenerate conics (cases 3a, 3b, 3c, 3d, 3e, 3f, 3g), 
\item a twisted cubic and one of its chords (cases 4a, 4b, 4c), 
\item an irreducible quartic curve. 
\end{enumerate}
In order to analyse these five cases in more detail, we call points and lines \emph{rational} if they lie in $\PG(3,q)$, and \emph{non-rational} if they are in an extension if $\PG(3,q)$.  

Note that as $\S=\H_1\cap\H_2$, the cubic extension of $\S$ to $\PG(3,q^3)$ contains the three conjugate points $P,P^q,P^{q^2}$ defining $\pi_\infty$. Also note that the   three points  $P,P^q,P^{q^2}$ do not lie on the extension of any rational line of $\PG(3,q)$. These two properties allow us to eliminate many of the cases in the above list.

We first consider case 1:  suppose $\S$ is the set of points on four lines. 
As  the points $P,P^q,P^{q^2}$ lie  in the cubic extension of $\S$, but do not lie on any rational lines, $\S$ contains at least three distinct non-rational lines, moreover these lines lie in the cubic extension $\PG(3,q^3)$. The configurations in \cite{BH} in case 1 consist of either rational lines; or pairs of lines which are conjugate over $\Fqq$; or 4-sets of lines which are conjugate over $\Fqqqq$. Thus $\S$ is not any of the configurations in case 1, so $\S$ is not four lines. 

Now consider case 2. The description in \cite{BH} shows that cases 2b, 2d only occur if the pencil of quadrics contains at most one hyperbolic quadric, so these cases do not occur for $\S$ (which is the base of a pencil that contains at least two hyperbolic quadrics, namely $\H_1,\H_2$).
  In cases 2a, 2c, $\S$ is the set of points lying on two rational lines $\ell,m$ and a rational non-degenerate conic $\C$. 
As $P,P^q,P^{q^2}$ lie in the extension of $\S$ and do not lie on any rational line, they lie on the cubic extension of $\C$. Hence $\C$ lies in 
the plane $\pi_\infty$. So $\S$ has $2q-1,2q$ affine points in cases 2c, 2a respectively, hence $|\D_1\cap\D_2|\in\{2q,2q+1\}$. If $q=2$, then $\D_1\cap\D_2$ is two $\Fq$-lines. If  $q\geq3$, then by Result~\ref{125}, the rational  lines $\ell,m$ correspond to $\Fq$-lines in the intersection 
 $\D_1 \cap \D_2 $; and by Lemma~\ref{117}, these two lines belong to different families in $\D_1$ (and belong to different families  $\D_2$).

Now consider case 3: suppose that $\S$ is the set of points on two non-degenerate conics $\C,\C'$. So $\C,\C'$  are either both rational, or they are a conjugate pair over $\Fqq$. 
As $P,P^q,P^{q^2}$ lie in the cubic extension of $\S$, without loss of generality assume that $P$ lies in the cubic extension   of $\C$.
Let $\alpha$ denote the  collineation of $P\Gamma L(4,q^3)$ induced by the field automorphism  $X^\alpha=X^q$.
The map $\alpha$ fixes $\PG(3,q)$,  so fixes the curve $\S$. Hence $\alpha$ either fixes both $\C$ and $\C'$, or interchanges $\C$ and $\C'$. The latter is not possible as $\alpha$ has order 3. Hence $\alpha$  fixes both $\C$ and $\C'$. Thus they are both rational conics and so cases 3d, 3e, 3f are not possible. 
 As $\alpha$ fixes $\C$,  we have $P^\alpha=P^q$ (and similarly $P^{q^2}$) lies in the cubic extension   of $\C$. That is, $\C$ lies in $\pi_\infty$. In cases 3a, 3b, 3c, 3g we have   $|\C\cap \C'|=x$ where $x=0,1,2,q+1$ respectively. Hence $\S$ contains $q+1, q, q-1,0$ affine points respectively. That is, if case 3 occurs, then $|\D_1\cap\D_2|\in \{q+2,q+1,q,1\}$. Further if $q\geq3$, then as $\S$ does not contain a line or a twisted cubic,  $\D_1\cap\D_2$  does not contain an $\Fq$-line by Result \ref{125}.

 Now consider case 4: suppose $\S$ is the set of points on a twisted cubic $\N$ and a chord $\ell$. As  the points $P,P^q,P^{q^2}$ lie  in the cubic extension of $\S$, but do not lie on any rational lines, they must lie on the cubic extension of $\N$. 
 So $\pi_\infty\cap\N$ does not contain a rational point. As $\ell$ does not lie in $\pi_\infty$, in cases 4a, 4b, 4c, $\S$ contains $2q-1,2q,2q+1$ affine points respectively. Hence $|\D_1 \cap \D_2 |\in\{2q,2q+1,2q+2\}$. 
If $q=2$, then $\D_1\cap\D_2$ is two $\Fq$-lines. If  $q\geq3$,  then by Result~\ref{125}, $\D_1\cap\D_2$ is two $\Fq$-lines, which  by Lemma \ref{117}  belong to different families in $\D_1$ (and belong to different families  $\D_2$).

Finally we consider case 5: suppose that 
$\S$ is an irreducible quartic curve. The plane $\pi_\infty$ meets $\S$ in four points (possibly repeated, possibly in an extension). As  $P,P^q,P^{q^2}$ are on the cubic extension of $\S$, 
 $\pi_\infty\cap\S$  contains exactly one rational point.
Hence $|\D_1 \cap \D_2| = |\S|$. As discussed in \cite{BH}, we can project the quartic curve $\S$ onto a plane, so a bound on the number $n=|\S|$ is given by the Hasse-Weil theorem, namely $
 (\sqrt q-1)^2 \leq n \leq (\sqrt q+1)^2 
$ (see \cite[Appendix IV]{Hirsch2} for details). Further, $\S$ does not contain a line or a twisted cubic, so  if $q\geq3$, then  by Result~\ref{125}, $\D_1\cap\D_2$ does not contain an $\Fq$-line. 

In summary, we have shown  the following are the possible intersection sizes when $\D_1\cap\D_2\neq\emptyset$.
\begin{itemize}
\item $|\D_1\cap\D_2|\in\{2q,2q+1,2q+2\}$, then $\D_1\cap\D_2$ is two $\Fq$-lines. If $q\geq3$,  these belong to different families of $\D_1$ (and belong to different families  $\D_2$). 
\item $|\D_1\cap\D_2|\in\{1,q,q+1,q+2\}$ and  if $q\geq3$, then $\D_1\cap\D_2$ does not contain an $\Fq$-line.
\item $|\D_1\cap\D_2|=n$ where $q-2\sqrt q+1 \leq n \leq q+2\sqrt q+1$, and  if $q\geq3$, then $\D_1\cap\D_2$ does not contain an $\Fq$-line.
\end{itemize}
This completes the proofs of part \ref{113a} (except the statement that $n\neq 2,3$) and part \ref{113b}. Part \ref{113c} follows by noting that 
  if $q=5$, then $q+2\sqrt q+1 < 2q+1$; and if $q>5$, then   $q+2\sqrt q+1 <2q$.

The remaining statement to prove is that  two maximum scattered linear sets cannot meet in exactly 2 or 3 points.  The remainder of this proof will prove this statement.  
Denote the conic  $\H_i\cap\pi_\infty$ by $\C_i$, $i=1,2$. We consider two cases, depending on whether $\C_1,\C_2$ are equal or distinct.
Suppose $\C_1=\C_2$, so $\H_1\cap\H_2$ contains a non-degenerate conic $\C_1=\C_2$. 
The base $\S$ contains a non-degenerate conic  in the following cases: 3g, 3c, 3b, 3a, 2c, 2a which give  $|\D_1\cap\D_2|$ equal $1,q,q+1,q+2, 2q,2q+1$ respectively. So we only need consider case 3c when $q=2,3$ and case 3b when $q=2$. Looking at the table in \cite{BH}, we see in case 3c when $q=2,3$ and case 3b when $q=2$, the associated pencil does not contain two hyperbolic quadrics. That is, the associated intersection sizes of $2,3$ do not occur. 
 We conclude that if $\C_1=\C_2$, then 
 $|\D_1\cap\D_2|\neq 2,3$. 

Now suppose $\C_1\ne \C_2$. As the cubic extensions of $\C_1,\C_2$ both contain the three conjugate points $P,P^q,P^{q^2}$, the two conics  meet in exactly one (rational) point, so let $X=\C_1\cap\C_2$.  We will show that in this case $|\D_1\cap\D_2|\geq4$, that is, we show that $\H_1\cap\H_2$ contains  at least 3  affine points.
  Denote the two lines of $\H_i$ through the point $X$ by $\ell_i,m_i$, $i=1,2$. 
Without loss of generality, assume $\ell_1\neq \ell_2$.
Consider the plane $\pi=\langle\ell_1,\ell_2\rangle$. We consider the three possibilities: either both $m_1,m_2$ are contained in $\pi$; neither $m_1,m_2$ are contained in $\pi$, or exactly one of $m_1,m_2$ is contained in $\pi$.
First, suppose both $m_1,m_2$ are contained in $\pi$, then $\pi$ is a tangent plane of both $\H_1$ and $\H_2$. Hence  the line $\pi\cap\pi_\infty$ is tangent to both $\C_1$ and $\C_2$ at the point $X$. As $\C_1,\C_2$ share the rational point $X$ and three conjugate points in the cubic extension, we have $\C_1=\C_2$, a contradiction. 
Next, suppose neither $m_1,m_2$ are contained in $\pi$. Then for $i=1,2$, $\pi\cap\H_i$ is two lines, one of which is $\ell_i$, let the other be $n_i$. So $n_i$ contains a point of $\C_i$ distinct from $X$, hence $n_1\neq n_2$.   In this case the three affine points $\ell_1\cap n_2,n_1\cap\ell_2,n_1\cap n_2$ are distinct and lie in $\H_1\cap\H_2$. Hence $|\D_1\cap\D_2|\geq4$ as required.
Finally, suppose  $m_1$ is contained in $\pi$ and $m_2$ is not contained in $\pi$.
Consider the plane $\alpha=\langle \ell_1,m_2\rangle$. As $m_2$ is not contained in $\pi$, we have $\pi\neq\alpha$. That is, $\alpha$ contains the lines $\ell_1,m_2$ but does not contain the lines $m_1,\ell_2$. Repeating the previous argument with $\alpha=\pi$ shows that $|\D_1\cap\D_2|\geq4$ as required. This completes the proof that  two maximum scattered linear sets cannot meet in exactly 2 or 3 points.  
\end{proof}

We can use the  proof of Theorem~\ref{113} to show that the following intersection sizes do occur. Note that this is not a complete list of existence, there a maximum scattered linear sets whose intersection corresponds to an irreducible quartic curve, and these are not included in the list below.

\begin{corollary} 
 In $\PG(1,q^3)$, there exists two maximum scattered linear sets  which meet in   $x$ points for each $x\in\{0,1,q,q+1,q+2,2q,2q+1,2q+2\}$, with the following exceptions for small $q$: $x=0,1,2,3,4$ when $q=2$; $x=1,3$ when $q=3$; and $x=4$ when $q=4$.
\end{corollary}

\begin{proof} As discussed in Section~\ref{sec2.3},  two maximum scattered linear sets with the same carriers are disjoint. So there do exist two maximum scattered linear sets with intersection size $0$ when $q\geq3$.

The proof of Theorem~\ref{113} shows that intersection sizes $1,q,q+1,q+2$ arise in cases 3g, 3c, 3b, 3a respectively; intersection size $2q$ arises in cases 2c and 4a; intersection size $2q+1$ arises in cases 2a and 4b; and intersection size $2q+2$ arises in case 4c. The table in \cite{BH} shows that there exists pencils with these bases for each case. However, for the base to arise as the intersection of two maximum scattered linear sets, we need the pencil to contain two hyperbolic quadrics. The table in \cite{BH} lists the number of hyperbolic quadrics in each pencil, and using this we see that there are some exceptions for small $q$. In particular, the following bases do not occur in a pencil that contains two hyperbolic quadrics: case 2c when $q=2$; case 3a when $q=2$, case 3b when $q=2,3$, case 3c when $q=2,3,4$; case 3g when $q=2,3$; case 4a when $q=2$. This gives the exceptions in the result.

When a pencil containing two hyperbolic quadrics with a certain base does exist, we can map it to an appropriate position to give our existence result. For example in case 4c, \cite{BH} gives the coordinates for a pencil whose base is a twisted cubic $\N$ and a line $\ell$ containing two points of $\N$.
Let $\alpha$ be a plane of $\PG(3,q)$ that contains no rational points of $\N$. So in the cubic extension, $\N$ meets $\alpha$ in three conjugate points $A, A^q,A^{q^2}$. Mapping $A, A^q,A^{q^2}$ to the three conjugate points $P,P^q,P^{q^2}$ defining $\pi_\infty$ maps the example from \cite{BH} to  an example of two hyperbolic quadrics that correspond to two maximum scattered linear sets meeting in $2q+2$ points.
\end{proof}

\begin{remark} {\textup{  Note that  the bound of $q\geq5$, $|\D_1\cap\D_2|\neq10$ given in Theorem~\ref{113}(\ref{113c}) is tight -- Magma \cite{magma} has been used to construct examples when $q=5$, of (a)   two maximum scattered linear sets that meet in 10 points with the intersection forming two $\Fq$-lines, and (b)  two maximum scattered linear sets that meet in 10 points with the intersection not containing any $\Fq$-lines.}}
\end{remark}
%

\subsection{The intersection of a maximum scattered linear set and a club}

%

We prove the following  about the intersection of a maximum scattered linear set and a club.

\begin{theorem}\label{114} In $\PG(1,q^3)$, let $\D_1$ be a maximum scattered linear set and let $\D_2$ be a club with head $H$. Then the following hold.
\begin{enumerate}
\item\label{114a}    $|\D_1\cap\D_2|\in\{0,1,2q,2q+1\} \cup \{n\,|\, q-2\sqrt{q} +1\leq n\leq q+ 2 \sqrt{q}+1  \mbox{ and } n \neq2 \}     $.

\item\label{114b} If  $q\geq3$ and $\D_1\cap\D_2$ contains an $\Fq$-line, then it is two $\Fq$-lines (which belong to different families of $\D_1$).
\item\label{114c} If  $q\geq5$ and  $|\D_1\cap\D_2|\in\{2q,2q+1\}$ and   $|\D_1\cap\D_2|\neq10$, then $\D_1\cap \D_2$ is two $\Fq$-lines.
\item\label{114d} If $H\in\D_1\cap\D_2$,
then   $|\D_1\cap\D_2|\in \{q,q+1,q+2,2q,2q+1\}$ (and $|\D_1\cap\D_2|\neq 2$). 
\end{enumerate}
\end{theorem}

\begin{proof} Suppose $|\D_1\cap\D_2|>0$. We consider two cases, depending on whether the head $H$ of the club $\D_2$ lies in the maximum scattered linear set $\D_2$ or not.

Case A: Suppose $H\in\D_1$, so $H\in\D_1\cap\D_2$. 
Let $P_\infty=H$ and consider the BB-representation of $\PG(1,q^3)$ in $\PG(3,q)$ with the plane at infinity denoted $\pi_\infty$. 
By Theorem~\ref{001}, $\D_1$ is represented by a hyperbolic quadric $\H$ meeting $\pi_\infty$ in a conic $\C$ whose cubic extension contains the three conjugate points defining the plane $\pi_\infty$; and
$\D_2$ is represented by a plane $\pi\neq\pi_\infty$.
We consider two cases, depending on whether $\pi$ is a tangent or secant plane of $\H$. Firstly, 
 suppose $\pi$ is a tangent plane to $\H$. Then the number of affine points in $\pi\cap\H$ is either  $2q-1,2q$ depending on whether   $\pi\cap \pi_\infty$ is either 
 secant or tangent  to $\C$. That is,   $|\D_1\cap \D_2 |$ is $2q,2q+1$  respectively. If $q=2$, then $\D_1\cap\D_2$ is two $\Fq$-lines.  If $q\geq3$, then by Result~\ref{125}, $\D_1\cap\D_2$ is two $\Fq$-lines which by Lemma~\ref{117} belong to different families of $\D_1$. 
 Next, suppose $\pi$ is a secant plane to $\H$. Then the number of affine points in $\pi\cap\H$ is either  $q-1,q,q+1$ depending on whether   $\pi\cap \pi_\infty$ is either secant, tangent or external to $\C$. That is, 
 $|\D_1\cap \D_2 | $ is $q,q+1,q+2$ respectively. Further if $q\geq3$, then by Result~\ref{125}, $\D_1\cap\D_2$ does not contain an   $\Fq$-line. 

Case B: Suppose $H\notin\D_1$, so $H\notin\D_1\cap\D_2$. Without loss of generality, assume $P_\infty\in\D_1\cap\D_2$ and 
consider the BB-representation of $\PG(1,q^3)$ in $\PG(3,q)$, with the plane at infinity  denoted $\pi_\infty$.  
By Theorem~\ref{001},  $\D_1$ corresponds   to a  hyperbolic quadric $\H$ that meets $\pi_\infty$ in a  conic whose cubic extension contains the three conjugate points defining the plane $\pi_\infty$; and $\D_2$ corresponds to a quadratic cone $\Q$ that meets $\pi_\infty$ in a  conic whose cubic extension contains the three conjugate points defining the plane $\pi_\infty$.   
The next part of the proof  is very similar to the proof of Theorem~\ref{113}. The intersection $\S=\H \cap \Q$   is the base curve of a non-singular pencil of quadrics, and the 
possibilities for $\S$ are detailed in \cite{BH}.  We only need consider pencils which contain both a hyperbolic quadric and a quadratic cone which eliminates some of the cases given in \cite{BH}.  Using the same notation for cases used in \cite{BH} (and in the proof of Theorem~\ref{113}),  the possibilities for $\S$ are:
\begin{enumerate}
\item four lines (does not occur)
\item two lines and a non-degenerate conic (case 2c), 
\item two  non-degenerate conics (cases 3a, 3b, 3c, 3d, 3e, 3f, 3g), 
\item  a twisted cubic and one of its chords (cases 4a, 4b), 
\item  an irreducible quartic curve. 
\end{enumerate}

If case 2c occurs, then $\S$ is the set of points lying on two rational lines and a rational non-degenerate conic. The same argument as that given in the proof of Theorem~\ref{113} shows that  $|\D_1\cap\D_2|=2q+1$ and  $\D_1\cap\D_2$ is two $\Fq$-lines (which belong to different families of $\D_1$ if $q\geq3$). 
If case 3 occurs, then  the same argument as that given in the proof of Theorem~\ref{113} shows that $|\D_1 \cap \D_2|\in \{q+2,q+1,q,1\}$ and if $q\geq3$, then $\D_1\cap\D_2$ does not contain an $\Fq$-line. 
If case 4a or 4b occurs, then the same argument as that given in the proof of Theorem~\ref{113} shows that $|\D_1\cap\D_2|\in\{2q,2q+1\}$ and  then $\D_1\cap\D_2$ is two $\Fq$-lines; if $q\geq3$, these belong to different families of $\D_1$ (and belong to different families  $\D_2$). 
If case 5 occurs, then the same argument as that given in the proof of Theorem~\ref{113} shows that if $n=|\D_1\cap\D_2|$, then $q-2\sqrt{q}+1\leq n\leq q+2\sqrt{q}+1$ and if $q\geq3$, then $\D_1\cap\D_2$ does not contain an $\Fq$-line.

In summary, we have shown  the following are the possible intersection sizes when $\D_1\cap\D_2\neq\emptyset$.
\begin{itemize}
\item If $H\in\D_1$, then one of the following occur.
\begin{itemize}
\item[{\scriptsize{$\bullet$}}] $|\D_1\cap\D_2|\in\{q,q+1,q+2\}$ and if $q\geq3$, then $\D_1\cap\D_2$ does not contain an $\Fq$-line.
\item[{\scriptsize{$\bullet$}}]  $|\D_1\cap\D_2|\in\{2q,2q+1\}$, then $\D_1\cap\D_2$ is two $\Fq$-lines.
\end{itemize}
\item If $H\notin\D_1$, then   one of the following occur.
\begin{itemize}
\item[{\scriptsize{$\bullet$}}]  $|\D_1\cap\D_2|\in\{1,q,q+1,q+2\}$ and if $q\geq3$, then $\D_1\cap\D_2$ does not contain an $\Fq$-line.
\item[{\scriptsize{$\bullet$}}]  $|\D_1\cap\D_2|\in\{2q,2q+1\}$ and  $\D_1\cap\D_2$ is two $\Fq$-lines.
\item[{\scriptsize{$\bullet$}}]  $|\D_1\cap\D_2|=n$ where $q-2\sqrt q+1 \leq n \leq q+2\sqrt q+1$, and if $q\geq3$, $\D_1\cap\D_2$ does not contain an $\Fq$-line.
\end{itemize}
\end{itemize}
This completes the proofs of part \ref{114a} (except the statement that $n\neq 2$) and parts \ref{114b}, \ref{114d}. Part \ref{113c} follows by noting that 
  if $q\geq5$ and $|\D_1\cap\D_2|\neq10$, then $q+2\sqrt q+1 < 2q$.

The remaining statement to prove is that a maximum scattered linear set and a club cannot meet in exactly 2 points. The remainder of this proof proves this statement. 
 First note that if $q=2$, then $\PG(1,2^3)$ contains 9 points, a maximum scattered linear set contains 7 points, and a club contains 5 points; so a maximum scattered linear set and a club to meet in at least 3 points. Using the above  summary,  the only way to get intersection size 2 with $q\neq2$ is in the last item.  This occurs when: $H\notin\D_1$; $\D_1$ corresponds   to a  hyperbolic quadric $\H$ that meets $\pi_\infty$ in a  conic $\C_1$ whose cubic extension contains the three conjugate points defining the plane $\pi_\infty$;  $\D_2$ corresponds to a quadratic cone $\Q$ with vertex $V\notin\H$ that meets $\pi_\infty$ in a  conic $\C_2$ whose cubic extension contains the three conjugate points defining the plane $\pi_\infty$; and the 
 curve $\S=\H\cap\Q$ is irreducible. That is,    the conics $\C_1$ and $\C_2$ are not equal. As the cubic extension of $\C_1,\C_2$ contains the three conjugate points $P,P^q,P^{q^2}$, the two conics  meet in exactly one (rational) point, so let $X=\C_1\cap\C_2$. There are two lines of $\H$ through $X$, denoted $\ell_1,\ell_2$, moreover the  plane $\pi_1=\langle \ell_1,\ell_2\rangle$  meets $\pi_\infty$ in a tangent line to $\C_1$. There is a unique plane $\pi_2$ through the line $VX$ that meets $\Q$ in exactly the line $VX$, so $\pi_2\cap\pi_\infty$ is   tangent to $\C_2$. As the 
(distinct) conics $\C_1,\C_2$ share one rational point and three conjugate points in the cubic extension, they do not share a tangent through $X$. Hence $\pi_1\cap\pi_\infty\neq\pi_2\cap\pi_\infty$ and so
  $\pi_1\ne \pi_2$.  Consider the two (distinct) planes $\alpha_i=\langle \ell_i,VX\rangle$, $i=1,2$. For $i=1,2$, as $\alpha_i\cap\H$ contains the line $\ell_i$, it contains a second line $m_i$ (with $X\notin m_i$). So $m_i\cap VX$  is a point $P_i$ which is distinct from $X,V$. As $VX$ is not a line of $\H$, it is a secant of $\H$, and so   $P_1=P_2$.  So the lines $m_1,m_2$ lie on a plane $\beta$ say. Now $\beta$ does not contain the vertex $V$ of the cone,  hence meets $\Q$ in a non-degenerate  conic $\E$.  The lines $m_1,m_2$ are distinct and contain the point  $P_1\in\E$. Hence at least one is a secant line of $\E$, and meets $\E$ in a further point $Q$. Thus $\H\cap\Q$ contains at least two affine points, namely $P_1,Q$, and so  $|\D_1\cap\D_2|\geq3$. 
This completes the proof that a maximum scattered linear set and a club cannot meet in exactly 2  points. 
\end{proof}

Note that Theorem~\ref{114} shows that the maximum intersection size bound of $2q+2$ given in  \cite[Lemma 2.4]{ferr03}  is not tight. The techniques used in the proof of Theorem~\ref{113} also show that the following intersection sizes do occur; in particular, 
the maximum intersection size  bound of $2q+1$ is tight.  
(Note that this is not a complete list of existence - intersections corresponding to an irreducible quartic curve are not included in the list below.)

\begin{corollary} 
\begin{enumerate}
\item In $\PG(1,q^3)$, there exists a maximum scattered linear set   and a club  which meet in 
$x$ points for each $x\in\{q,q+1,q+2,2q,2q+1\}$, with the exception of $x=2$ when $q=2$. 
\item  In $\PG(1,q^3)$, there  exists a maximum scattered linear set   and a club  which meet in exactly $1$ point iff $q$ is odd (in which case the point of intersection is not the head of the club).
 \end{enumerate}
\end{corollary}

\begin{proof}
Consider the proof of Case A in Theorem~\ref{114}, where the intersection numbers 
  $\{q,q+1,q+2,2q,2q+1\}$ occurs as the intersection of a plane and a quadratic cone. All these cases occur except the case when $q=2$ and the intersection is size 2. This intersection number arises in case A as follows:  when $\pi$ is a secant plane to $\H$, and $\pi\cap\pi_\infty$ is a secant line. That is, there are two secant planes of $\H$ through the secant line $\pi\cap\pi_\infty$, namely $\pi,\pi_\infty$. This is a contradiction when $q=2$ since a secant line of a hyperbolic quadric $\H$ lies in exactly three planes of $\PG(3,2)$, one plane is secant to $\H$ and two planes are tangent to $\H$. We observe that each of these intersection numbers also occurs in case B. In particular this means that for each intersection number, there are examples where the head of the club lies both in, and not in, the maximum scattered linear set.

For part 2: it follows from  the proof of Theorem~\ref{114} that the only case where a maximum scattered linear set and club  meet in exactly one point is in case B, when the hyperbolic quadric $\H$ and the quadratic cone $\Q$ meet in a repeated conic  (which lies in  $\pi_\infty$). 
This is case 3g in \cite{BH}.  The table in \cite{BH} shows that there exists  a pencil  containing one hyperbolic quadric and one quadratic cone whose base is one non-degenerate conic iff $q$ is odd.  We can  map the example given in \cite{BH} to get an example of a hyperbolic quadric and a quadratic cone  that meet in a non-degenerate conic  $\C$, such that the  cubic extension  of $\C$ contains  the three conjugate points  defining $\pi_\infty$. This corresponds to an example of a  maximum scattered linear set and a club meeting in exactly one 1 point existing iff $q$ is odd.  
\end{proof}

\begin{remark} {\textup{ Note that the bound of $q\geq5$, $|\D_1\cap\D_2|\neq10$ given  in part Theorem~\ref{114}(\ref{114c}) is tight --  Magma \cite{magma} has been used to construct examples when $q=5$, of (a) a maximum scattered linear set and a club that meet in 10 points with the intersection forming two $\Fq$-lines, and (b)  a maximum scattered linear set and a club that meet in 10 points with the intersection not containing any $\Fq$-lines.}} \end{remark}

\subsection{The intersection of two clubs}

%
%
%

We prove the following   about the intersection of two clubs.

\begin{theorem}\label{116} 
Let $\D_1,\D_2$ be two distinct clubs in $\PG(1,q^3)$.
Then the following hold.
\begin{enumerate}
\item\label{116a}  $|\D_1\cap\D_2|\in\{0,1,2,2q\} \cup \{n\,|\, q-2\sqrt{q} +1\leq n\leq q+ 2 \sqrt{q}+1 \} $.

\item\label{116b}  If $q\geq3$ and $\D_1\cap\D_2$ contains an $\Fq$-line and a further point, then $\D_1\cap\D_2$ is two $\Fq$-lines.  In this case, $\D_1,\D_2$ have distinct heads. 

\item\label{116c} If  $q>5$ and $|\D_1\cap \D_2|=2q$, then $\D_1\cap\D_2$ is two $\Fq$-lines. 
\item\label{116d}  If $\D_1$ and $\D_2$ share a common head, then $\D_1\cap\D_2$ is either one point or an $\Fq$-line.
\end{enumerate}
\end{theorem}

\begin{proof}
Suppose $|\D_1\cap\D_2|>0$.
Denote the heads of $\D_1,\D_2$ by $H_1,H_2$ respectively. We consider
four cases: (A) when $H_1=H_2$; (B) when $H_1\neq H_2$ and  
$H_1,H_2\in\D_1\cap\D_2$; (C) when $H_1\neq H_2$ and  exactly one of
$H_1,H_2$ lies in $\D_1\cap\D_2$; and (D) when $H_1\neq H_2$ and
$H_1,H_2\notin\D_1\cap\D_2$.

Case A: Suppose $H_1=H_2$. We work in the BB representation of
$\PG(1,q^3)$ in $\PG(3,q)$  using $P_\infty=H_1$ as the point at infinity.
By Theorem~\ref{001}, $\D_1,\D_2$ are represented by distinct
affine planes $\pi_1,\pi_2$ respectively.
If $\pi_1\cap\pi_2$ is a line contained in $\pi_\infty$, then $|\D_1\cap
\D_2 | =1$ and $\D_1\cap\D_2=\{H_1\}=\{H_2\}$. If  $\pi_1\cap\pi_2$ is a  line
not contained in $\pi_\infty$, then
  $|\D_1\cap \D_2 | =q+1$. Further, if $q\geq3$, then by Result~\ref{125}, $\D_1\cap\D_2$
is an $\Fq$-line; moreover if $q=2$, then $|\D_1\cap \D_2 | =3$ and so $\D_1\cap\D_2$
is an $\Fq$-line.
Note that  both cases always occur for any $q$.

Case B: Suppose $H_1\neq H_2$ and  
$H_1,H_2\in\D_1\cap\D_2$.
We work
in the BB representation of $\PG(1,q^3)$ in $\PG(3,q)$  using $P_\infty=H_2$ as
the point at infinity.
By Theorem~\ref{001}, $\D_1$ is represented by a quadratic cone
$\Q$ that meets $\pi_\infty$ in a  conic whose cubic extension contains the three conjugate points defining the plane $\pi_\infty$;
and $\D_2$ is represented by an affine plane $\pi$.
As $H_1\in\D_2$, the  vertex $V$ of $\Q$ lies in $\pi$. 
Hence $\pi\cap\Q$ is either $V$, one line through $V$, or two lines through $V$ 
(depending on whether the line $\pi\cap\pi_\infty$ is  external, tangent or secant to the conic $\Q\cap\pi_\infty$). Hence $|\D_1\cap\D_2|$ is $2,q+1,2q$ respectively. Further, if $q\geq3$, then by 
Result~\ref{125},   then  $\D_1\cap \D_2 $ is respectively two points, one $\Fq$-line or two $\Fq$-lines. We also note this also holds for $q=2$.
Moreover, each case exists for all $q$.

Case C: Suppose $H_1\neq H_2$ and  exactly one of
$H_1,H_2$ lies in $\D_1\cap\D_2$.
Without loss of generality suppose $H_2\in\D_1\cap\D_2$. 
We work
in the BB representation of $\PG(1,q^3)$ in $\PG(3,q)$  using $P_\infty=H_2$ as
the point at infinity.
By Theorem~\ref{001}, $\D_1$ is represented by a quadratic cone
$\Q$ that meets $\pi_\infty$ in a  conic whose cubic extension contains the three conjugate points defining the plane $\pi_\infty$;
and $\D_2$ is represented by an affine plane $\pi$.
As $H_1\notin\D_2$,  $\pi$ does not contain the vertex of $\Q$. Hence $\pi\cap \Q$ is a conic containing $q+1$, $q$ or $q-1$ affine points
(depending on whether the line $\pi\cap\pi_\infty$ is  external, tangent or secant to the conic $\Q\cap\pi_\infty$). Hence  $\D_1\cap \D_2$ contains $q+2,q+1,q$ points respectively. Further, if $q\geq3$, then by Result~\ref{125}, this intersection does not contain an $\Fq$-line. Moreover, each case exists for all $q$.

Case D: Suppose $H_1\neq H_2$ and $H_1,H_2\notin\D_1\cap\D_2$.  Let
$P_\infty\in\D_1\cap\D_2$ and  work in the BB representation of $\PG(1,q^3)$ in
$\PG(3,q)$  using $P_\infty$ as the point at infinity.
By Theorem~\ref{001}, $\D_1$, $\D_2$ are represented by
quadratic cones $\Q_1,\Q_2$ respectively.
Let $\S=\Q_1\cap \Q_2$. In this case, we cannot use the characterisation of
the base of a non-singular pencil of quadrics given in \cite{BH} -- the quadrics
$\Q_1,\Q_2$ are both singular, so we may be working with a pencil which is
singular. In particular, the breakdown of cases given in \cite{BH} is not complete for singular pencils.

As $\S$ is the intersection of two quadrics, it is a curve of degree $4$. If the curve $\S$ is irreducible, then  the same
explanation as given in the proof of Theorem~\ref{113} shows that the
number $n=|\S|=|\D_1\cap\D_2|$ satisfies $(\sqrt q-1)^2 \leq n \leq
(\sqrt q+1)^2$.

Suppose $\S$ is 
reducible, then its degree can be partitioned in four ways:
$4=1+1+1+1=1+1+2=2+2=1+3$.
Now the vertex of the cones $\Q_1,\Q_2$ correspond to the heads $H_1,H_2$ respectively.
As $H_1,H_2\notin\D_1\cap\D_2$, the set $\S$ does not contain the vertex
of either cone. In particular this means that $\S$ does not contain a
generator line from either cone. That is, $\S$ does not contain a line,
and so if the curve $\S$ is reducible, then the degree can be partitioned
in exactly one way, namely $4=2+2$.  
That is, the points of $\S$ lie on two non-degenerate conics
$\C_1,\C_2$. Following the same argument as in  case 3 in the proof of
Theorem~\ref{113}, we deduce that
one of the conics, say $\C_1$, is contained in  $\pi_\infty$ and  both $\C_1$ and $ \C_2$ are rational conics. There are now two cases to  consider.
If $\C_2$ also lies in $\pi_\infty$, then $\S\subset\pi_\infty$ and $\S=\Q_1\cap \Q_2$, so $\S$ is a repeated non-degenerate conic, that is, $\C_1=\C_2$.  In this case $\D_1\cap\D_2$ is a point which is distinct from $H_1,H_2$.
Note that by Lemma \ref{118}  this subcase occurs iff $q$ is even. 
If $\C_2$ lies a plane $\pi\neq\pi_\infty$, then $|\C_1\cap
\C_2|$ is either $0,1,2$ and so  $\S$ contains $q+1, q, q-1$ affine points
respectively. That is, $|\D_1\cap\D_2|\in \{q+2,q+1,q,1\}$. Further if $q\geq3$, then as
$\S$ does not contain a line or a twisted cubic,  $\D_1\cap\D_2$  does
not contain an $\Fq$-line by Result \ref{125}.

In summary, we have shown the following are the possible intersection sizes  when $\D_1\cap\D_2\neq\emptyset$.
\begin{itemize}
\item If $H_1=H_2$, then $|\D_1\cap\D_2|\in \{1,q+1\}$ and $\D_1\cap\D_2$  is either a point or an $\Fq$-line. Moreover all cases exist.
\item If $H_1\neq H_2$, and $\D_1\cap\D_2$ contains both heads, then $|\D_1\cap\D_2|\in\{2,q+1,2q\}$ and $\D_1\cap\D_2$ is either two points, an $\Fq$-line or two $\Fq$-lines.  Moreover all cases exist.

\item If  $H_1\neq H_2$, and $\D_1\cap\D_2$ contains exactly one head, then  $|\D_1\cap\D_2|\in\{q,q+1,q+2\}$ and if $q\geq3$, then $\D_1\cap\D_2$ does not contain an $\Fq$-line.  Moreover all cases exist.

\item If  $H_1\neq H_2$, and $\D_1\cap\D_2$ contains neither head, then one of the following occur.
\begin{itemize}
\item[{\scriptsize{$\bullet$}}]  $|\D_1\cap\D_2|\in\{1,q,q+1,q+2\}$ and if $q\geq3$, then $\D_1\cap\D_2$ does not contain an $\Fq$-line. Moreover, the intersection size 1 exists in this case iff $q$ is even.
\item[{\scriptsize{$\bullet$}}]  $|\D_1\cap\D_2|=n$ where $q-2\sqrt q+1 \leq n \leq q+2\sqrt q+1$, and if $q\geq3$, then  $\D_1\cap\D_2$ does not contain an $\Fq$-line.
\end{itemize}
\end{itemize}
This completes the proof of parts \ref{116a}, \ref{116b}, \ref{116d}. Noting that  if $q>5$, then $q+2\sqrt q+1 < 2q$ completes the proof of part \ref{116c}.
\end{proof}

The proof of Theorem~\ref{116} proves the existence of  the following examples. In particular, we note that the maximum size given in Theorem~\ref{116} is attained. 

\begin{corollary} In $\PG(1,q^3)$, the following examples exist for all $q\geq2$.
\begin{enumerate}

\item\label{116attained} There exists two clubs that meet in $x$ points for $x\in\{1,2,q,q+1,q+2,2q\}$. 
\item \label{116attained2} There exists two clubs that meet in one $\Fq$-line (this occurs with two clubs with a common head, and with two clubs with distinct heads).
\end{enumerate}
\end{corollary}

Further,  the proof of Theorem~\ref{116}  shows the following.

\begin{corollary} In $\PG(1,q^3)$:
\begin{enumerate}
\item If  $\D_1,\D_2$ are clubs with $|\D_1\cap\D_2|=2$, then $\D_1,\D_2$ have distinct heads and $\D_1\cap\D_2$ consists of the two  heads.  
\item If $q$ odd and $\D_1,\D_2$ are clubs with $|\D_1\cap\D_2|=1$, then $\D_1,\D_2$ have a common head. 
\end{enumerate}
\end{corollary}
We note that, if $q$ even, then   intersection size 1 occurs  with two clubs with a common head, and with two clubs with distinct heads.

\begin{remark} {\textup{Note that the bound  of $q\geq5$, $|\D_1\cap\D_2|\neq10$ given in part Theorem~\ref{116}(\ref{116c}) is tight -- Magma \cite{magma} has been used to construct examples when $q=5$, of (a) two clubs that meet in 10 points with the intersection forming two $\Fq$-lines, and (b)  two clubs that meet in 10 points with the intersection not containing any $\Fq$-lines. }}  
\end{remark}

\section{Conclusion}

We conclude with a discussion on the existence of disjoint linear sets of rank 3 in $\PG(1,q^3)$. There exists disjoint maximum scattered linear sets  in $\PG(1,q^3)$, $q\geq3$, namely   two maximum scattered linear sets with the same carriers are disjoint. 
Further, as noted in the introduction, 
\cite[Cor 6.6]{SVDVV} shows that   there exists disjoint clubs  in $\PG(1,q^3)$, $q$ odd; and 
 \cite[Remark 3.13]{SVDVV} conjectures that there do not exist disjoint clubs  in $\PG(1,q^3)$,   $q$ even. 

The remaining case to consider is whether there exists a maximum scattered linear set $\D_1$ and a club $\D_2$   in $\PG(1,q^3)$ with $|\D_1\cap\D_2|=0$. Checking with Magma \cite{magma} shows it is not possible to find a disjoint maximum scattered linear set   and   club   when $q=2,3$, however, when $q\geq4$, we think there will always exist an example of a maximum scattered linear set $\D_1$ and a club $\D_2$ with $|\D_1\cap\D_2|=0$. One approach to look at this in the $\PG(3,q)$ BB-representation  using \cite[Thm 7.2]{sher86} which shows that there is always a club with head $P_\infty=(1,0)$ that is disjoint from the  cubic surface $\S(1,0,\rho,\sigma)$ where $\rho\sigma\neq0$. The set $\S(1,0,\rho,\sigma)$ is either a point, a club, a maximum scattered linear set, or a Sherk surface of size $q^2+1$ depending on the choice of $\rho,\sigma$. 
A choice of $\rho,\sigma$ for which $\S(1,0,\rho,\sigma)$ is a maximum scattered linear set will give an example of a club and a maximum scattered linear set which are disjoint. We have constructed many values which work for various values of $q$, but have not found a universal choice for $\rho,\sigma$ which works for all $q$.

%
%
%
%
%
%
%

\end{document}